\documentclass{amsart}
\usepackage{amsmath,amsfonts,amssymb,amsthm}
\usepackage{xcolor,mhsetup,bm}
\usepackage[extra,safe]{tipa}
\usepackage{mathtools}
\usepackage{tikz}
\usepackage{amssymb}

\usetikzlibrary{positioning}
\usetikzlibrary{shapes.geometric, arrows}
\usetikzlibrary{calc,decorations.pathmorphing,shapes,arrows}
\tikzset{
	photon/.style={decorate, decoration={snake}, draw=red},
	electron/.style={draw=blue, postaction={decorate},
		decoration={markings,mark=at position .55 with {\arrow[draw=blue]{>}}}},
	gluon/.style={decorate, draw=magenta,
		decoration={coil,amplitude=4pt, segment length=5pt}} 
}

\usepackage[applemac]{inputenc}
\usepackage{mathrsfs}
\usepackage{mathabx}
\def \d{\mathrm{d}} 
\newtheorem{theorem}{Theorem}
\newtheorem{definition}[theorem]{Definition}

\newtheorem{lem}[theorem]{Lemma}

\newtheorem{corol}[theorem]{Corollary}
\newtheorem{remark}[theorem]{Remark}

\def\I{\mathsf{1}}

\def \N{\mathbb{N}}

\def\my_c{c_\infty}

\newcommand{\mynewtheorem}[2]{
	\newaliascnt{#1}{dummy}
	\newtheorem{#1}[#1]{#2}
	\aliascntresetthe{#1}
	\expandafter\def\csname #1autorefname\endcsname{#2}
}

\newcommand{\be}{\begin{equation}}
	\newcommand{\ee}{\end{equation}}
\newcommand{\bde}{\begin{displaymath}}
	\newcommand{\ede}{\end{displaymath}}
\newcommand{\beq}{\begin{eqnarray*}}
	\newcommand{\eeq}{\end{eqnarray*}}
\newcommand{\beqa}{\begin{eqnarray}}
	\newcommand{\eeqa}{\end{eqnarray}}
\newcommand{\bel }{\left\{\begin{array}{ll}}
	\newcommand{\eel}{\cr \end{array} \right.}

\newcommand{\seq}[1]{{\lbrace #1 \rbrace}}

\newcommand{\E}{\mathbb{E}}

\def\rr{{\mathbb R}}

\def\P{{\mathbb P}}

\def\0{{\mathbf{0}}}

\begin{document}

	\title[IBP formula for exit times]{Integration by parts formula for exit times of one dimensional diffusions}\author{Noufel Frikha, Arturo Kohatsu-Higa and Libo Li}
	\thanks{The research of A. Kohatsu-Higa was supported by KAKENHI grant 20K03666. He would like to express his thanks to the organizers of the meeting ``Kolmogorov Operators and Their Applications'' for their kind invitation to contribute to this volume.  The research of N. Frikha has benefited from the support of the Institut Europlace de Finance.\\
		Noufel Frikha.  Universit\'e Paris 1 Panth\'eon-Sorbonne, Centre d'Economie de la Sorbonne, 106 Boulevard de l'H\^{o}pital, 75642 Paris Cedex 13, France. E-mail: noufel.frikha@univ-paris1.fr.\\
		Arturo Kohatsu-Higa.  Department of Mathematical Sciences
		Ritsumeikan University, 1-1-1
		Nojihigashi, Kusatsu, Shiga, 525-8577, Japan. E-mail: khts00@fc.ritsumei.ac.jp.\\
		Libo Li. University of New South Wales, School of Mathematics and Statistics. Sydney NSW 2052 Australia.  E:mail: libo.li@unsw.edu.au}
	
	\begin{abstract}
		In line with the methodology introduced in \cite{frikha:kohatsu:li} for formulating probabilistic representations of integration by parts involving killed diffusion, we establish an integration by parts formula for the first exit time of one-dimensional diffusion processes. However, our approach diverges from the conventional differential calculus applied to the associated space Markov chain; instead, we employ calculus techniques that focus on the underlying time variables.  
	\end{abstract}
	\maketitle
	\section{Introduction}

	In this article, we consider the process denoted as $X$, which is defined as the unique strong solution to the following one dimensional stochastic differential equation (SDE for short) with the following dynamics:
	\begin{align}
		\label{sde:dynamics}
		X_t=x+\int_0^tb(X_s) \, \d s+\int_0^t\sigma(X_s)\, \d W_s,
	\end{align}
	\noindent where the coefficients $b, \sigma: \mathbb{R} \rightarrow \mathbb{R}$ are smooth and bounded functions and $(W_t)_{t\geq0}$ denotes a standard one-dimensional Brownian motion within a given filtered probability space $(\Omega, \mathcal{F}, (\mathcal{F}_t)_{t\geq0}, \mathbb{P})$. 
	
	The aim of the present paper is to derive an integration by parts (IBP) formula for the first hitting time of a given level $L$ by the process $X$. To elaborate further, for a fixed starting point $x\geq L$, we define $\tau$ as the first time $X$ crosses the level $L$: $\tau= \inf\{t \geq 0: X_t=L\}$. For any given finite time horizon $T>0$, our focus here is on establishing a explicit probabilistic representation for the IBP formula pertaining to the quantity:
	$$
	\mathbb{E}[f'(\tau \wedge T)]  
	$$
	
	\noindent where $f$ is a real-valued smooth function defined on $\mathbb{R}_+$.

	
	Numerous theoretical properties concerning these exit times have been documented. While not exhaustive, one can refer, for instance, to \cite{kent1980eigenvalue} and \cite{kent1982spectral}. Nonetheless, as mentioned in \cite{pitmanyor}:
	``The distribution of the exit time can be expressed using the sum of an infinite sequence of independent exponential variables. While such expansions are of interest in a number of contexts, the corresponding representations of the density or cumulative distribution function of the exit time can be difficult to work with''.

	The application of Malliavin Calculus to random variables associated with exit times remains somewhat fragmented. In \cite{Bismut} and \cite{BAKS}, the authors establish certain regularity properties of the stopped diffusion process at the boundary of a general smooth domain. In contrast, in \cite{AM}, the authors prove that, in general, exit times can not be Malliavin differentiable even once. Therefore, to the best of our knowledge, the pursuit of an integration by parts formula for exit times using Malliavin Calculus appears to be a daunting task.
	
	In this article, we first establish a probabilistic representation formula in the following format
	\begin{align*}
		\E[f(\tau \wedge T)]=\E[f(\bar{\tau} \wedge T)\prod_{i=1}^{N_T+1}\theta^i].
	\end{align*} 
	Here $ \bar{\tau} $ is an approximation of the exit time $ \tau $ and the sequence of random variables $(\theta^i)_{1\leq i \leq N_T+1}$ is derived from a Markov chain that approximates the trajectory of the underlying diffusion process. Additionally, $(N_t)_{0\leq t\leq T}$ represents a Poisson process, and $f$ is a bounded measurable function. The proof employs similar arguments as those utilized in \cite{frikha:kohatsu:li}.
	
	Our starting point in establishing our IBP formula is based on the probabilistic representation formula mentioned earlier. The representation proposed in this context is associated with a general perturbation theory of Markov processes, commonly employed for investigating the smoothness properties of probability density functions. This theory, which can be seen as an intermediate step between approximations and their limits, has been extensively developed in various directions. without being exhaustive, references to this theory can be found in works such as \cite{FP}, \cite{ballybook} and \cite{FKL1, 10.1214/19-AIHP992, frikha:kohatsu:li, chen:frikha:li, CHAUDRUDERAYNAL20221, CHAUDRUDERAYNAL20211}.

	To derive IBP formulas in our context, we initially observe that the problem bears a certain 'dual' relationship with the one pertaining to one-dimensional killed diffusion processes.  In a loose sense, space and time variables undergo an interchange. With this perspective in mind, our approach involves considering the conditional distribution of the underlying time variables with respect to the space variables. Subsequently, we apply the IBP formulas with respect to the time variables. Given that these formulas involve differentiation of the time random variables, we also anticipate that the degree of degeneracy will be more pronounced compared to the case of killed diffusion.
	
	When we carry out the conditioning process with respect to the space variables as described above, the distribution of the jump times undergoes a transformation from exponential to generalized inverse Gaussian. This change occurs because, in addition to the exponential density, the spatial Gaussian transition densities must also be taken into account as components of the conditional distribution of the jump times. Consequently, based on the preceding discussion, it becomes necessary to introduce dual operators that are associated with the jump times of Poisson processes conditioned on space variables.

	The literature regarding integration by parts formulas concerning the jump times of a Poisson process is extensive. In all cases known to the authors, the general approach involves employing IBP formulas with respect to exponential distributions, which may inherently introduce boundary conditions. Dealing with these terms can pose difficulties, necessitating the use of localization techniques centered around the boundary values of these time random variables to address this technical challenge. For further insights into this topic, please consult references such as \cite{Denis}, \cite{BallyBavouzet}, and \cite{BallyClement}.

	In the current configuration, we do not seek to circumvent these boundary terms; rather, we include them in our calculations. Since the calculation is carried out conditionally, based on the values of space variables, within a random partition of the time interval-assumed, without loss of generality, to be all distinct-we ascertain that all boundary terms will ultimately disappear. 
	
	Lastly, following a similar approach to our companion paper \cite{FKL1}, where we employed a time stabilization argument for stopped processes, in this article, we employ a spatial stabilization technique. This approach leads to a Malliavin variance quantity that relies on the squares of distances between points within the underlying Markov chain structure. The central result of this paper is presented in Theorem \ref{th:8}, where we establish the IBP formula for the first hitting time of $X$. We achieve this using the probabilistic representation derived in Theorem \ref{th:st1}
	
	%
	%
	%
	%
	
	\vskip2pt
	\noindent 	{\it Notations:  }
	We use $ A\stackrel{\E}{=} B$ to denote that $ \E[A]=\E[B] $. We will be working on a Wiener space with all necessary additional independent r.v.'s added to it. Products $\prod_{j=i}^k\theta^j $ are  defined to be $ 1 $ when $ i>k $. We also use $ \bar{\mathbb{N}}=\mathbb{N}\cup\{0\} $ and $ \bar{\mathbb{N}}_{n+1}=\{0,...,n+1\} $.
	
	
	\section{Preliminaries}
	\label{sec:pre}

	\subsection{Assumptions}
	Throughout the article, we work on the filtered probability space $(\Omega, \mathcal{F}, (\mathcal{F}_t)_{t\geq0}, \mathbb{P})$ which is assumed to be rich enough to support all random variables that will be considered in what follows. In addition, we will suppose that the following assumptions on the coefficients are in force:
	\subsection*{Assumption (H)}
	\begin{itemize}
		\item[(i)] The coefficients of SDE \eqref{sde:dynamics} are smooth, bounded with bounded derivatives. That is, $b\in \mathscr{C}^{\infty}_b(\rr)$ and $ a:=\sigma^2 \in \mathscr{C}^{\infty}_b(\rr)$. 
		\item[(ii)] The function $\sigma$ is uniformly elliptic and bounded. That is, there exist $\underline{a}, \overline{a}>0$ such that for any $x\in \rr$, $\underline{a} \leq a(x)\leq \overline{a}$.
	\end{itemize}
	
	\subsection{A reflection principle}

	As in \cite{FKL1}, our probabilistic representation involves the following approximation process 
	\begin{align}
		\label{eq:defbX}
		\bar{X}^{s, x}_{t}\equiv \bar{X}^{s, x}_{t}(\rho) = \rho x+(1-\rho)(2L-x)+\sigma(x)(W_t-W_s),
	\end{align} 
	\noindent where $ \rho $ is a Bernoulli random with parameter $1/2$, independent of $ W $, namely $\P(\rho=1) = \P(\rho=0) = 1/2 $. 
	This random variable is related to a reflection principle. 
	We may use the simplified form $ \bar{X}_t$ for $\bar{X}^{0,x}_{t} $. The above approximation process $ \bar{X} $ is the main building block in the forthcoming probabilistic representation formula of the couple $(\tau \wedge T, X_{\tau\wedge T})$. The next lemma follows from the reflection principle, see e.g. Karatzas and Shreve \cite{Karatzas1991}, or from the explicit density of the Brownian motion killed at level $L$ given for $ x, y >L$ by  
	\begin{align*}
		\frac{1}{\sqrt{2\pi \sigma^2 T}}\left(\exp\left(-\frac{(y-x)^2}{2 \sigma^2 T}\right) - \exp\left(-\frac{(y-(2L-x))^2}{2 \sigma^2 T}\right)  \right).
	\end{align*}
	\begin{lem} 
		\label{lem:1}Consider the process
		\begin{align*}
			\bar{Y}_t=x+\sigma(x)W_t, \quad x\geq L,
		\end{align*}
		together with its hitting time $ \bar{\tau}:=\inf\left\{t \geq 0 :\bar{Y}_t=L\right\} $. Then, for any bounded and measurable function $f$, it holds that
		\begin{align}
			\E\left[f(\bar{Y}_T)\I_\seq{\bar{\tau} >T}\right]=&
			\E\left[f(\bar{Y}_T)\I_\seq{\bar{Y}_T\geq L}\right]-
			\E\left[f(2L-\bar{Y}_T)\I_\seq{\bar{Y}_T<L}\right] = 2\E\left[(2\rho-1)f(\bar{X}_T)\I_\seq{\bar{X}_T\geq L}\right].\label{eq:refa}
		\end{align}
	\end{lem}
	

	
	
	\subsection{The underlying Markov chain and associated simplified Malliavin Calculus}
	\label{sec:3a}
	
	This section serves as an introduction to the foundational Markov chain that will underpin our probabilistic representation and facilitate the development of our IBP formulas. Furthermore, we shall present the essential material required for our differential calculus computations. It is noteworthy that a substantial portion of this section has been drawn directly from \cite{FKL1} for reference.
	
	We consider a Poisson process with parameter $\lambda>0$, independent of the one-dimensional Brownian motion $W$ with jump times $T_i$ and we set $ \zeta_i:= T_i\wedge T$, for $ i\in\mathbb{N} $, with the convention that $\zeta_0= T_0 = 0$. 
	We let $ \pi $ be the partition of $ [0,T]$ given by $\pi:=\{0=:\zeta_0<\cdots<\zeta_{N_T} \leq T\} $. Associated with this set, we define the simplex $ A_{n}:=\{t\in (0,T]^n;0<t_1<\cdots<t_n\leq T\}$. For instance, on the set $ \{N_T=n\} $, $ n\in\mathbb{N} $, we have $ (\zeta_1,\dots,\zeta_n)\in A_n $ and $\zeta_{n+1} = T$.  In particular, for the set $ \{N_T=0\} $ (i.e. $ n=0 $), we let $ \pi:=\{0,T\} $ and $ A_0=\emptyset $. In this sense, we will use through the rest of the paper the index $ n\in\bar{\mathbb{N}} $ without any further mention of its range of values.
	
	
	

	Let $\bar{X}:= (\bar{X}_i)_{ i \in\bar{\mathbb{N}}} $ be the discrete time Markov chain starting at time $0$ from $\bar{X}_0 = x$ and evolving according to 
	\begin{align}
		\label{eq:MCa}
		\bar{X}_{i+1}:=&\rho_{i+1}\bar{X}_i+(1-\rho_{i+1}) (2L-\bar{X}_i)+\sigma_i Z_{i+1}, \,\quad  i \in\bar{\mathbb{N}}_ {N_T},
	\end{align} 
	where for simplicity we set $ \sigma_i:=\sigma(\bar{X}_i) $, $ Z_{i+1}:= W_{\zeta_{i+1}} - W_{\zeta_i}=\sigma_i^{-1}\left(\bar{X}_{i+1}-\rho_{i+1}\bar{X}_i-(1-\rho_{i+1})(2L-\bar{X}_i)\right)$ and $\{\rho_i;i\in\mathbb{N}\} $ is an i.i.d. sequence of Bernoulli$(1/2)$ random variables such that $W, \, N $ and $\{\rho_i;i\in\mathbb{N}\} $ are mutually independent. In what follows, we use the notation $ h_i\equiv h(\bar{X}_i)$, $ i\in\bar{\mathbb{N}}_{n+1} $ for any function $ h:\mathbb{R}\rightarrow\mathbb{R} $. In particular, the reader may have noticed that we already used this notation in the above formula for $ h=\sigma $. We also associate to the Markov chain $\bar{X}$ the following sets
	\begin{equation}
		\label{set:Din:X}
		D_{i,n}:=\{\bar{X}_{i}\geq L, N_T=n\} , \quad \mbox{ for } i\in \bar{\N}_{n+1}.
	\end{equation}

	Instead of using an infinite dimensional calculus as it is usually done in the literature, the approach developed below is based on a finite dimensional calculus for which the dimension is given by the number of jumps of the underlying Poisson process involved in the Markov chain $\bar{X}$.
	
	We now introduce the following space of smooth random variables.
	
	\begin{definition}
		\label{def:1}
		For $ i\in\bar{\mathbb{N}}_n$, we let $ {\mathbb{S}}_{i+1,n}(\bar{X}) $ be the subset of random variables $ H\in \mathbb{L}^0 $ such that there exists a measurable function $ h:\mathbb{R}^2\times \{0,1\}\times A_2\rightarrow \mathbb{R} $ satisfying
		\begin{enumerate}
			\item $\displaystyle{ H=
				h(\bar{X}_i,\bar{X}_{i+1},\rho_{i+1},\zeta_i,\zeta_{i+1}) 
			}
			$ on the set $\{N_T=n\} . $ 
			\item For any $ (s,t)\in A_2 $ and $ r\in\{0,1\} $, $ h(\cdot,\cdot,r,s,t)\in\mathscr{C}_p^\infty(\mathbb{R}^2) $.
		\end{enumerate}
	\end{definition} 
	
	For a r.v. $ H\in  {\mathbb{S}}_{i+1,n}(\bar{X})$, $ i\in\bar{\mathbb{N}}_n$, we may sometimes simply write 
	\begin{equation}
		\label{abuse:notation:space}
		H\equiv H(\bar{X}_i,\bar{X}_{i+1},\rho_{i+1}, \zeta_i,\zeta_{i+1}),
	\end{equation}
	that is the same symbol $ H $ may denote the r.v. or the function in the set $ {\mathbb{S}}_{i+1,n}(\bar{X}) $.  Also note that $ {\mathbb{S}}_{n+1,n}(\bar{X}) $ is well defined.
	
	One can easily define the flow derivatives for $ H\in\mathbb{S}_{i+1,n}(\bar{X}) $ as follows:
	\begin{align}
		\partial_{\bar{X}_{i+1}}H:=&\partial_2h
		(\bar{X}_i,\bar{X}_{i+1},\rho_{i+1},\zeta_i,\zeta_{i+1}), \nonumber\\
		\partial_{\bar{X}_{i}}H:=&\partial_1h(\bar{X}_i,\bar{X}_{i+1},\rho_{i+1},\zeta_i,\zeta_{i+1})+\partial_2h(\bar{X}_i,\bar{X}_{i+1},\rho_{i+1},\zeta_i,\zeta_{i+1})\partial_{\bar{X}_{i}}\bar{X}_{i+1}, \label{eq:flow} \\
		\partial_{\bar{X}_{i}}\bar{X}_{i+1}:=&
		(2\rho_{i+1}-1)+\sigma'_i Z_{i+1}.
		\nonumber
	\end{align}	
	\noindent We now define the derivative and integral operators for $ H\in {\mathbb{S}}_{i+1,n}(\bar{X}) $, $ i\in\bar{\mathbb {N}}_n $, as 
	\begin{align}
		\label{eq:I1}
		\mathcal{I}_{i+1}(H)
		:= &H \frac{ Z_{i+1}}{\sigma_i(\zeta_{i+1}-\zeta_i) }- {\mathcal{D}_{i+1} H}, \qquad \mathcal{D}_{i+1}H:=\partial_{\bar{X}_{i+1}}H.
	\end{align}

	Note that due to the above definitions and Assumption $ \mathbf{(H)} $, we also have that  $ \mathcal{I}_{i+1}(H),\mathcal{D}_{i+1}H \in{\mathbb{S}}_{i+1,n}(\bar{X})$ so that we can define iterations of the above operators, namely $ \mathcal{I}_{i+1}^{\ell+1}(H) = \mathcal{I}_{i+1}(\mathcal{I}^{\ell}_{i+1}(H))$ and similarly $\mathcal{D}^{\ell+1}_{i+1}(H) = \mathcal{D}_{i+1}(\mathcal{D}^{\ell}_{i+1}H)$, $ \ell\in\bar {\mathbb{N}} $, with the convention $\mathcal{I}^{0}_{i+1} (H) = \mathcal{D}^{0}_{i+1}(H) = H$.
	
	We define the filtration $\mathcal{G} := (\mathcal{G}_i)_{i\in \bar{\mathbb{N}}}$ where $\mathcal{G}_i := \sigma(Z^i,\zeta^i,\rho^i) $ with the notation $ a^i :=(a_1,\dots,a_i) $ for $ a=Z,\ \zeta,\ \rho $, $ i \in\mathbb{N}$. We let $ \mathcal{G} _0$ be defined as the trivial $ \sigma $-field. 
	We assume that the $ \sigma $-fields $\mathcal{G}_i$ are complete. Throughout this article, we will use the following notation for a certain type of conditional expectation that will appear frequently.
	For any $X\in \mathbb{L}^{1}$ and any $ i\in \bar{\mathbb{N}}_{n}$, 
	\begin{align*}
		\E_{i,n}[X]:=\E[X\,\vert\,\mathcal{G}_i,T^{n+1},\rho^{n+1}, N_T=n].
	\end{align*}
	Here again we have used the vector notation $ S^{n+1}=(S_1,....,S_{n+1}) $ for $ S=T$, $\rho $.
	Having the above definitions at hand, the following duality formula\footnote{This duality is obtained using the Gaussian density of $ \bar{X}_{i+1} $ while in classical Malliavin calculus it is based on the density of the Wiener process. Therefore the derivative and integral, $ \mathcal{D}_{i+1} $ and $ \mathcal{I}_{i+1} $, $ i\in\bar{\mathbb{N}}_n $ defined in the  formula \eqref{eq:IBP} are renormalizations of the usual duality principle in Malliavin calculus. In our case this notation simplifies greatly many equations.} is satisfied for any $ f\in\mathscr {C}^1_p(\rr)$ and any $ (i,\ell) \in\bar{\mathbb{N}}_n \times \mathbb{N}$:
	\begin{align}
		\label{eq:IBP}
		\E_{i,n}\left[{\mathcal{D}}^\ell_{i+1}f(\bar{X}_{i+1})H\right]=&\E_{i,n}\left[f(\bar{X}_{i+1}){\mathcal{I}}^\ell_{i+1}(H)\right].
	\end{align}
	In order to obtain explicit norm estimates for random variables in $ {\mathbb{S}}_{i+1,n}(\bar{X}) $, it is useful to define for $H\in \mathbb{S}_{i+1,n}(\bar{X})$, $ i\in\bar{\mathbb{N}}_n $ and $p\geq1$
	$$
	\|H\|_{p, i, n}^p:=\E_{i,n}\left[|H|^p\right].
	$$
	
	We will also use at several places the following \emph{extraction formula} for $ H_1,H_2 \in\mathbb{S}_{i+1,n}(\bar{X})$ :
	\begin{align}
		\label{eq:exta}
		\mathcal{I}^\ell_{i+1}(H_1H_2)=\sum_{j=0}^{\ell}(-1)^{j}\binom{\ell}{j}
		\mathcal{I}_{i+1}^{\ell-j}(H_1)\mathcal{D}^{j}_{i+1}H_2,
	\end{align}
	\noindent whose proof follows by induction. By iteration, one also obtains that $ {\mathcal{I}}^\ell_{i}(1)\in \mathbb{S}_{i,n}(\bar{X}) $ and it satisfies ${\mathcal{I}}^{\ell+1}_{i}(1)={\mathcal{I}}^\ell_{i}(1){\mathcal{I}}_{i}(1) -\ell {\mathcal{I}}^\ell_{i}(1)$ which, in particular, implies:
	\begin{align}
		\label{eq:link:integral:hermite:pol}
		&{{\mathcal{I}}}_{i}^\ell(1)=(-1)^\ell \mathcal{H}_\ell({a_{i-1}}(\zeta_i-\zeta_{i-1}),\sigma_{i-1} Z_i),
	\end{align}
	
	\noindent where $\mathcal{H}_\ell$ stands for the Hermite polynomial of order $\ell$, $\ell \in \mathbb{N}$, defined by $\mathcal{H}_\ell(t, x)= (g( t, x))^{-1} \partial^{\ell}_x g(t,x)$. Here, $g(t,x)$ denotes the density of the Gaussian law $\mathcal{N}(0,t)$.
	Using \eqref{eq:exta} for $ H_1=1 $ and $ H_2=H $ as well as \eqref{eq:I1}, the following norm bound for stochastic integrals is clearly satisfied for $ i\in \bar{\mathbb{N}}$ and any measurable set $ A\in\mathcal{F} $
	\begin{align}
		\label{eq:Hesta}
		\left\|\I_{A}{{\mathcal{I}}}_{i+1}^{\ell}(H)\right\|_{p, i, n}\leq C_{\ell,p}\sum_{j=0}^{\ell}(\zeta_{i+1}-\zeta_i)^{\frac{j-\ell}{2}}\|\I_{A}\mathcal{D}^j_{i+1}H\|_{p, i, n}.
	\end{align}
	
	We will frequently use the following H\"older inequality for smooth random variables $ H_1,H_2 \in\mathbb{S}_{i+1,n}(\bar{X})$: for any $ i\in \bar{\mathbb{N}} $, any $ p,p_1,p_2\geq 1 $ satisfying $ p^{-1}=p_1^{-1}+p_2^{-1} $ and any $ A\in\mathcal{F} $:
	\begin{align}
		\label{eq:Hest1}
		\|\I_{A}H_1H_2\|_{p, i, n}\leq& C\|\I_{A}H_1\|_{p_1, i, n}\|\I_{A}H_2\|_{p_2,i,n}.
	\end{align}

	Quantities such as $ \E_{i,n}[\delta_L(\bar{X}_{i+1})H] $ for $ H\in\mathbb{S}_{i+1,n}(\bar{X}) $ have a clear meaning due to the IBP formula \eqref{eq:IBP} (see also the theory introduced in Chapter V.9 \cite{IW} for a much more general framework) or in the sense of conditional laws
	
	\begin{align*}
		\E_{i,n}[\delta_L(\bar{X}_{i+1})H] =&\E_{i,n}[\I_{D_{i+1,n}}{\mathcal{I}}_{i+1}(H)] = \E_{i,n}[H\vert \bar{X}_{i+1}=L]g(a_i(\zeta_{i+1}-\zeta_i),L-\bar{X}_i),\nonumber
	\end{align*}
	\noindent where the set $D_{i, n}$ is defined by \eqref{set:Din:X}. We finally introduce the following space of random variables which satisfies certain time (ir)regularity estimates.
	\begin{definition}
		For $\ell\in\mathbb{Z}$, $ i\in \bar{\mathbb{N}}_n$, the space $ {\mathbb{M}}_{i+1,n}(\bar{X},\ell/2) $ is the set of random variables $ H\in\mathbb{L}^0 $ satisfying the property {\it (1)} in Definition \ref{def:1} and such that 
		\begin{align*}
			\I_{D_{i,n}}\|\I_{D_{i+1,n}}H\|_{p, i, n}\leq C(\zeta_{i+1}-\zeta_i)^{\ell/2},
		\end{align*}
		{for some deterministic constant $ C $ independent of $ (p,i,n) $.}
	\end{definition}

	Again we remark that the definition of the space $ {\mathbb{M}}_{i+1,n}(\bar{X},\ell/2)  $ uses the conditional norm $\E_{i, n}[.]$ and property {\it (1)} in Definition \ref{def:1}, so that this property is always stated on $\left\{N_T = n \right\}$, $n\in \N$.
	
	
	A straightforward consequence of equation \eqref{eq:exta} and \eqref{eq:Hesta} is the following property.
	%

	\begin{lem}
		For $ j\in\{0,1\} $, $ i\in \bar{\mathbb{N}}, $ and $ k\in\mathbb{N} $, if $ H_1\in{\mathbb{M}}_{i+1,n}(\bar{X},j/2)\cap\mathbb{S}_{i+1,n}(\bar{X}) $ with $ \|\mathcal{D}_{i+1}^kH_1\|_{p, i, n}\leq C $ then $ {\mathcal{I}}_{i+1}^k(H_1)\in{\mathbb{M}}_{i+1,n} (\bar{X},(j-k)/2)$. Furthermore, if $ H_2\in  {\mathbb{M}}_{i+1,n}(\bar{X},k/2) $ then the product $ H_1H_2\in {\mathbb{M}}_{i+1,n}(\bar{X},(j+k)/2)  $.
	\end{lem}

	\begin{lem}\label{chain:rule}
		Let $H\equiv H(\bar X_i, \bar X_{i+1}, \rho_{i+1}, \zeta_i,\zeta_{i+1}) \in \mathbb{S}_{i+1,n}(\bar X)$ with 
		$ i\in \bar{\mathbb{N}}, $ then the following chain rule type formula holds
		\begin{gather*}
			\partial_{\bar X_i}\mathcal{I}_{i+1}(H) = \mathcal{I}_{i+1}(\partial_{\bar X_i}H) - \frac{\sigma'_i}{\sigma_i}\mathcal{I}_{i+1}(H) .
		\end{gather*}
	\end{lem}
	\begin{proof}
		From the extraction formula \eqref{eq:exta}, the usual chain rule and the fact that $\partial_{\bar X_i}\mathcal{I}_{i+1}(1) = -\frac{\sigma'_i}{\sigma_i}\mathcal{I}_{i+1}(1)$, we have that
		\begin{align*}
			\partial_{\bar X_i} \mathcal{I}_{i+1}(H)
			& = -\frac{\sigma'_i}{\sigma_i}\mathcal{I}_{i+1}(1) H + \mathcal{I}_{i+1}(1) \partial_{\bar X_{i}} H - \partial_{\bar X_{i}} {\mathcal{D}}_{i+1}H\\
			\mathcal{I}_{i+1}(\partial_{\bar X_i} H)
			& = \mathcal{I}_{i+1}(1) \partial_{\bar X_i} H - {\mathcal{D}}_{i+1}\partial_{\bar X_i} H.
		\end{align*}
		Note that $\partial_{\bar X_i}$ and ${\mathcal{D}}_{i+1}$ do not commute. Indeed, by the usual chain rule, one has
		\begin{align*}
			\partial_{\bar X_{i}} {\mathcal{D}}_{i+1}h -  {\mathcal{D}}_{i+1}\partial_{\bar X_i} h & = \partial_1\partial_2 h  + \partial^2_2 h \frac{\partial  \bar X_{i+1}}{\partial \bar X_i} - {\mathcal{D}}_{i+1}(\partial_1h + \partial_2 h\frac{\partial  \bar X_{i+1}}{\partial \bar X_i} )= -\partial_2 h  \frac{\sigma'_i}{\sigma_i}
		\end{align*}
		where in the last equality we used the fact that ${\mathcal{D}}_{i+1}\frac{\partial  \bar X_{i+1}}{\partial \bar X_i}  = \frac{\sigma'_i}{\sigma_i}$. By combining the above computations we obtain
		\begin{align*}
			\partial_{\bar X_i} \mathcal{I}_{i+1}(H) - \mathcal{I}_{i+1}(\partial_{\bar X_i} H) & = - \frac{\sigma'_i}{\sigma_i}\left(\mathcal{I}_{i+1}(1) - {\mathcal{D}}_{i+1} H\right) = -\frac{\sigma'_i}{\sigma_i}\mathcal{I}_{i+1}(H).
		\end{align*}
	\end{proof}

	\section{The probabilistic representation for the first hitting time}
	\label{sec:5}
	\subsection{Preliminaries}

	The probabilistic representation formula for the first hitting time $ \tau $ of $X$ is obtained in a similar way as in the case of the killed process treated in \cite{frikha:kohatsu:li}. We briefly explain the main arguments and deliberately omit some technical details.

	In order to introduce the associated semigroup, we let $\tau^{s,x}:=\inf\{t \geq s : X^{s,x}_t=L\}$ where $ X^{s,x} $ stands for the unique solution to
	\eqref{sde:dynamics} starting from $ x $ at time $ s $ and for any real-valued bounded and measurable function $ f $ defined on $\mathbb{R}_+\times \mathbb{R}$, we let
	$  P_t f(u,x)=\E\left[f(u+\tau^{u,x}\wedge t,X^{u,x}_{u+\tau^{u,x}\wedge t})\right] $. We will sometimes omit the dependence of $(\tau, X)$ with respect to $(u,x)$ when there is no confusion and simply write $P_t f(u,x) = \mathbb{E}[f(u + \tau\wedge t, X_{\tau \wedge t}) ]$. 
	
	Under Assumption $ \mathbf{(H)} $, for $ f $ a smooth function, we will assume for the sake of the argument that  $ Pf \in \mathscr{C}^{1,2}(\mathbb{R}_{+} \times (\mathbb{R}_+ \times [L,\infty)))$ with $\sup_{0\leq t \leq T} |\partial^{\ell}_x P_t f|_\infty \leq C$, for $\ell=1, 2$, for some positive constant $C:=C(T, f, a, b)$ and that it satisfies $ \partial_t P_t f={\mathcal{L}} P_t f$ on $ [L,\infty) $, $ t>0 $, where the differential operator $ {\mathcal{L}}$ is given by 
	$$
	\mathcal{L} f(u, x)= \partial_u f(u, x)  + b(x) \, \partial_x f(u,x) +\frac 12 a(x)\, \partial_x^2 f(u, x).
	$$  
	Removing the above hypothesis requires some technical arguments which can be found in Proposition 3.1 in \cite{FKL1}.


	
	The argument used to obtain the probabilistic representation formula starts by applying It\^o's formula to $(P_{T-t} f(u + t\wedge \bar{\tau}, \bar{Y}_{t\wedge\bar\tau}))_{t\in [0,T]}$, recalling that $\bar{Y}_t = x + \sigma(x) W_t$ and $ \bar{\tau} $ is its associated exit time from $[L, \infty)$. Hence, from It\^o's rule and Lemma \ref{lem:1}, it holds
	\begin{align*}
		& f(T \wedge \bar{\tau}, \bar{Y}_{T \wedge \bar{\tau}}) \\
		& \stackrel{\E}{=} P_{T}f(u, x) +\int_0^T \left(-\partial_r P_{r}f(u + s, \bar{Y}_s)\Big|_{r=T-s } + \partial_u P_{T-s} f(u+ s, \bar{Y}_s) +\frac{1}{2}a(x)\partial_x^2 P_{T-s} f(u+ s, \bar{Y}_s)\right)\I_\seq{\bar{\tau}>s} \, \d s\\
		& \stackrel{\E}{=} P_{T}f(u, x) +\int_0^T \left(\frac{1}{2}\left(a(x)-a(\bar{Y}_s)\right)\partial_x^2P_{T-s} f(u+ s, \bar{Y}_s)-b(\bar{Y}_s)\partial_xP_{T-s}f(u+ s, \bar{Y}_s)\right)\I_\seq{
			\bar{\tau}>s} \, \d s \\
		& \stackrel{\E}{=} P_{T}f(u, x) + 2(2\rho-1)\int_0^T \left(\frac{1}{2}\left(a(x)-a(\bar{X}_s)\right)\partial_x^2P_{T-s} f(u+s, \bar{X}_s)-b(\bar{X}_s)\partial_x P_{T-s}f(u+s, \bar{X}_s)\right)\I_\seq{\bar{X}_s\geq L} \, \d s.
	\end{align*}
	
	We now rewrite the previous representation using the Markov chain $(\bar{X}_i)_{0\leq i \leq N_T+1}$ defined by \eqref{eq:MCa} together with the Poisson process $N$. From the previous identity, we get
	\begin{align}
		P_T f(u, x) & \stackrel{\E}{=}  f(u+ T \wedge \bar{\tau}, \bar{Y}_{T\wedge \bar{\tau}}) \nonumber \\
		& + 2(2\rho-1)\int_0^T \left(\frac{1}{2}\left(a(\bar{X}_s) -a(x)\right)\partial_x^2P_{T-s} f(u+s, \bar{X}_s) +b(\bar{X}_s)\partial_x P_{T-s}f(u+s, \bar{X}_s)\right)\I_\seq{\bar{X}_s\geq L} \, \d s \label{integral:time:first:step} \\
		& \stackrel{\E}{=}  f(u+ T \wedge \bar{\tau}, \bar{Y}_{T\wedge \bar{\tau}}) e^{\lambda T}\I_\seq{N_T=0}  + e^{\lambda T} \lambda^{-1} 2 (2\rho_{N_T}-1) \nonumber \\
		& \quad \times \left\{\frac12(a(\bar{X}_1) - a(x)) \partial^2_x P_{T- \zeta_1}f(u + \zeta_1, \bar{X}_1) + b(\bar{X}_1) \partial_x P_{T-\zeta_1}f(u+\zeta_1, \bar{X}_1) \right\} \I_\seq{\bar{X}_1 \geq L} \I_\seq{N_T=1}. \nonumber
	\end{align}
	
	We now apply the IBP formula \eqref{eq:IBP} with respect to the random variables $ \bar{X}_1$ for the last term appearing in the right-hand side of the above expression. The IBP formula is applied once to the terms associated with the drift coefficient $ b $ and two times with respect to the terms related to the diffusion coefficient $ a$. In order to do that one first has to take the conditional expectation $\E_{0,1}[.]$. However, one must be cautious insofar as these IBPs involve the indicator function $ \I_\seq{\bar{X}_1\geq L} $. Hence,
	\begin{align*}
		P_{T}f(u, x)
		\stackrel{\E}{=}&  f(u+ T \wedge \bar{\tau}, \bar{Y}_{T\wedge \bar{\tau}}) e^{\lambda T} \I_\seq{N_T=0}
		+ e^{\lambda T} \lambda^{-1} 2(2\rho_{N_T}-1)
		\\
		\times&  \left(\frac{1}{2}{{\mathcal{I}}}_{1}\left(\left(a(\bar{X}_1) - a(x)\right)
		\I_\seq{ \bar{X}_1\geq L}\right)\partial_x P_{T-\zeta_1}f(\bar{X}_1) + {{\mathcal{I}}}_{1}\left(b(\bar{X}_{1})\I_\seq{ \bar{X}_1\geq L}\right) P_{T-\zeta_1}f(\bar{X}_{1})\right) \I_\seq{N_T= 1} \\
		\stackrel{\E}{=}&  f(u+ T \wedge \bar{\tau}, \bar{Y}_{T\wedge \bar{\tau}}) e^{\lambda T} \I_\seq{N_T=0}+ e^{\lambda T} 2\lambda^{-1}(2\rho_{N_T}-1)\\
		& \times \left(\frac{1}{2}{{\mathcal{I}}}_{1}\left(\left(a(\bar{X}_1) - a(x)\right)
		\right)\partial_x P_{T-\zeta_1}f(u+\zeta_1, \bar{X}_1) + {{\mathcal{I}}}_{1}\left(b(\bar{X}_{1} ) \right) P_{T-\zeta_1}f(u+\zeta_1, \bar{X}_{1})\right) \I_\seq{ \bar{X}_1\geq L} \I_\seq{N_T= 1},
	\end{align*}
	
	\noindent where we used the extraction formula \eqref{eq:exta} applied to the r.v. $ \I_\seq{\bar{X}_s\geq L} $ and Lemma 10.5 \cite{frikha:kohatsu:li} (taking first the conditional expectation w.r.t $\zeta_1$) for $ \ell=1,\ k=0 $ for the last equality. 
	
	We apply again the IBP formula \eqref{eq:IBP} in order to remove the first order derivative on $P_{T-\zeta_1}f$. We then use again the extraction formula  \eqref{eq:exta}. We obtain
	\begin{align}
		& P_{T}f(u, x) \stackrel{\E}{=} f(u+ T \wedge \bar{\tau}, \bar{Y}_{T\wedge \bar{\tau}}) e^{\lambda T} \I_\seq{N_T=0} + e^{\lambda T} 2\lambda^{-1}(2\rho_{N_T}-1)\nonumber\\
		& \times \left(\frac{1}{2}\mathcal{I}^2_{1} \left(a(\bar{X}_1) - a(x)\right) P_{T-\zeta_1}f(u+\zeta_1, \bar{X}_1) + {{\mathcal{I}}}_{1}\left(b(\bar{X}_{1} ) \right) P_{T-\zeta_1}f(u+\zeta_1, \bar{X}_{1})\right) \I_\seq{ \bar{X}_1\geq L} \I_\seq{N_T= 1} \label{prob:representation:step1}\\
		& - e^{\lambda T} 2\lambda^{-1}(2\rho_{N_T}-1) \frac12 \mathcal{I}^1_{1}\left(a(\bar{X}_1) - a(x)\right)P_{T-\zeta_1}f(u+\zeta_1, \bar{X}_{1}) \delta_L(\bar{X}_1)  \I_\seq{N_T= 1}.\nonumber
	\end{align}
	Using the extraction formula and then Lemma 10.5 in \cite{frikha:kohatsu:li}, we get
	\begin{align*}
		e^{\lambda T} &2\lambda^{-1} \mathbb{E}[(2\rho_{N_T}-1) \mathcal{I}^1_{1}\left(a(\bar{X}_1) - a(x)\right)P_{T-\zeta_1}f(u+\zeta_1, \bar{X}_{1}) \delta_L(\bar{X}_1)  \I_\seq{N_T= 1}] \\
		& = e^{\lambda T} \lambda^{-1} \mathbb{E}[(2\rho_{N_T}-1) (a(L)-a(x))\mathcal{I}^1_{1}\left(1\right)P_{T-\zeta_1}f(u+\zeta_1, \bar{X}_{1}) \delta_L(\bar{X}_1)  \I_\seq{N_T= 1}] \\
		& = \int_0^T \mathbb{E}[(2\rho-1) (a(L)-a(x)) \frac{(L-x)}{a(x) s} \delta_L(\bar{Y}_s) P_{T-s}f(u+s, L)] \, \d s \\
		& =- \frac{a(L)-a(x)}{a(x)} \mathbb{E}[f(u+\bar\tau,L)\I_\seq{\bar\tau\leq T}].
	\end{align*}
	Plugging the above identity into \eqref{prob:representation:step1}, we get
	\begin{equation}\label{prob:representation:before:iteration}
		\begin{aligned}
			P_{T}f(u, x) & \stackrel{\E}{=} f(u+ T \wedge \bar{\tau}, \bar{Y}_{T\wedge \bar{\tau}}) e^{\lambda T} \I_\seq{N_T=0} + e^{\lambda T} \bar{\theta}^1 \I_\seq{ \bar{X}_1\geq L} P_{T-\zeta_1}f(u+\zeta_1, \bar{X}_1) \I_\seq{N_T= 1} \\
			& \quad + \frac{a(L) - a(x)}{a(x)} f(u+\bar{\tau}, L) \I_\seq{\bar\tau\leq T},
		\end{aligned}
	\end{equation}
	\noindent where for $i=1, \cdots, n$ 
	\begin{equation}\label{definition:thetai}
		\bar{\theta}^i := 2\lambda^{-1}(2\rho_{i}-1) \left(\mathcal{I}^2_{i}(c_2^i) + {\mathcal{I}}_{i}(c_1^i) \right), 
	\end{equation}
	with
	\begin{equation}
		\label{definition:coefficients:c1:c2}
		\begin{aligned}
			c_1^i := b_i, \quad c_2^i  := \frac12 (a_i - a_{i-1}),
		\end{aligned}
	\end{equation}
	\noindent recalling that $a_i = a(\bar{X}_i)$ and $b_i = b(\bar{X}_i)$. Let us emphasize that, under assumption \textbf{(H)}, the following time estimates hold: for all $p\geq1$, there exists $C:=C(a, b, T, p)$ such that
	\begin{align}
		\label{eq:est2}
		\|\I_\seq{\bar{X}_1\geq L}\mathcal{I}_{1}(a(\bar{X}_1)-a(x))\|_{p, 0, 1} & \leq C,\\
		\|\I_\seq{\bar{X}_1\geq L}\mathcal{I}^2_{1}(a(\bar{X}_1)-a(x))\|_{p, 0, 1} +\|\I_\seq{\bar{X}_1\geq L}\mathcal{I}_{1}(b(\bar{X}_{1}))\|_{p, 0, 1}& \leq C \zeta_1^{-1/2}, \nonumber
	\end{align}
	
	\noindent which in turn, by using the fact that on $\P(N_T=1, \zeta_1 \in dt) = \lambda e^{-\lambda T} dt$ on $[0,T]$, lead to the integrability of the second term appearing in the right-hand side of \eqref{prob:representation:before:iteration}.

	The main idea to obtain the following result consists in iterating the identity \eqref{prob:representation:before:iteration} by following similar lines of reasoning as those employed in \cite{frikha:kohatsu:li,chen:frikha:li,FKL1}. Omitting the remaining technical details, we obtain the following result.   
	
	\begin{theorem}
		\label{th:st1}
		Let $T>0$ and suppose that Assumption \textbf{(H)} is satisfied. For any $n\in\mathbb{N}$, define on the set $ \{N_T=n\} $ the following random variables
		\begin{equation*}
			\Gamma_{n}= \left\{
			\begin{array}{ll}
				\prod_{{i}=1}^{n}  {\theta}^i
				& \mbox{ if } n \geq1, \\
				1 & \mbox{ if } n = 0,
			\end{array}
			\right.
		\end{equation*}

		\noindent and
		
		\begin{equation*}
			\bar{\Gamma}_{n}= 
			\frac{a(L) - a_{{n}}}{a_{{n}}} \Gamma_n,
		\end{equation*}
		
		\noindent where,
		\begin{align}
			\theta^i:=&\I_\seq{\bar{X}_i>L}\bar \theta^i
			=2(2\rho_i-1)\lambda^{-1}\left({{\mathcal{I}}}_{i}(c^i_1)+{{\mathcal{I}}}^2_{i}(c^i_2)\right)\I_\seq{\bar{X}_i>L}\notag\\
			=& \I_\seq{{\bar{X}_{i}>L}}
			2(2\rho_{i}-1)\lambda^{-1}\left\{
			c^{i}_2{\mathcal{I}}_{i}^2(1)+(c^{i}_1-2\mathcal{D}_ic^{i}_2){\mathcal{I}}_{i}(1)+\mathcal{D}^2_ic_2^{i}+\mathcal{D}_ic_1^{i}
			\right\}.
			\label{eq:ab}
		\end{align}

		Then, for any bounded and measurable function $f$  defined on $\rr_+ \times \rr$, the following probabilistic representation holds
		\begin{align}
			\E[f(T\wedge \tau, X_{T \wedge \tau})]  &= e^{\lambda T} \mathbb{E}\left[  f((\zeta_{N_T}+ {\bar{\tau}^{N_T}}) \wedge T, {\bar{Y}}_{(\zeta_{N_T}+ {\bar{\tau}^{N_T}}) \wedge T})   \Gamma_{N_T}\right] \nonumber\\
			& \quad + e^{\lambda T} {\mathbb{E}\left[f(\zeta_{N_T}+\bar{\tau}^{N_T}, L) \I_\seq{ { \zeta_{N_T}+\bar{\tau}^{N_T}} \leq T} \bar{\Gamma}_{N_T}\right]}.
			\label{eq:PRF}
		\end{align}
	\end{theorem}
	Importantly note that we have used the extraction formula in order to obtain \eqref{eq:ab}.  For this, we refer the reader to \eqref{eq:exta} and to the proof of Lemma 2.5 \cite{frikha:kohatsu:li} to see how it is applied.
	In order to avoid long expressions in the above identity, we have used the shortened notation $\bar{\tau}^{N_T}$ for ${\bar{\tau}^{\zeta_{N_T},\bar{X}_{N_T}}}$ where we recall that $ \bar{\tau}^{s,x} $ stands for the first hitting time of the level $L$ associated to the approximation process $ \bar{Y} $ starting from $x$ at time $s$. That is,
	\begin{align*}
		\bar{\tau}^{s,x}:=\inf\{u \geq s : \bar{Y}^{s,x}_u=L\}.
	\end{align*}

	In particular, $\bar{\tau}^{N_T}$ is sampled independently of $N$ and $W$. Similarly, we used the shorten notation $\bar{Y}_{(\zeta_{N_T}+ {\bar{\tau}^{N_T}})\wedge T}$ for $\bar{Y}^{\zeta_{N_T}, \bar{X}_{N_T}}_{(\zeta_{N_T}+{\bar{\tau}^{N_T}})\wedge T}$ where $\bar{Y}_t^{s, x}=x + \sigma(x) (W_t-W_s)$. We also recall that $ N $ is a $ \lambda $-Poisson process with jump times $ T_n$, $n\in\mathbb{N}$,
	so that $\Delta T_i =  T_i-T_{i-1} $ is a sequence of i.i.d. exponential random variables with parameter $ \lambda$ and that on the set $ \{N_T=n\} $, we have $ \pi=\{0=\zeta_0<...<\zeta_n<\zeta_{n+1}=T\} $ recalling that $\zeta_i = T_i \wedge T$.
	
	The second term on the right-hand side of \eqref{eq:PRF} containing $ \bar{\Gamma} $ stems from the boundary terms appearing when one performs the integration by parts formulas in the expansion. In particular, this term does not appear in the case of the killed diffusion investigated in \cite{frikha:kohatsu:li}.

	
	
	Note that both terms appearing on the right-hand side of \eqref{eq:PRF} have the same form if we assume that $ \zeta_{N_T}+\bar{\tau}^{N_T}<T $ {in which case ${\bar{Y}}_{(\zeta_{N_T}+ {\bar{\tau}^{N_T}}) \wedge T} = L$}. Also note that $ \mathbb{P}\left(\zeta_{N_T}+\bar{\tau}^{N_T}=T\right)=0 $. 
	
	
	%

	\section{IBP formulas with respect to the exit time}
	\subsection{IBP with respect to generalized inverse Gaussian times}
	
	In this section, we shall introduce the general transfer formula employing IBP with respect to the interarrival times of the Poisson process. A transfer formula, as explained in reference \cite{FKL1}, is a formula that transfers the derivative from outside the expectation to inside, making it subsequently amenable to iterative procedures.
	
	The principal difference lies in our current consideration of the conditional distribution of the random variables with respect to the spatial variables. Furthermore, given our perpetual focus on compact time intervals, it follows that test functions are bounded.

	As our objective is to elucidate the technique, as opposed to pursuing utmost generality and for the sake of streamlining expressions, we restrict our consideration to $ f(t, x) \equiv f(t) $, subject to the simplifying condition that $ f(T) = 0 $, where $ T >0 $ is an arbitrary time horizon. In other words, the test function $ f $ solely relies on the time variable and vanishes at the time interval's boundary. As is customary, in more general circumstances, one may consider $ \bar{f}(t) = f(t) - f(T) $ instead of $ f $.

	The overarching structure of the IBP formula, when dealing with the stopped process $ X_{T\wedge \tau} $, was derived through an initial conditioning on the Poisson process. Then, one performs IBP operations with respect to the spatial variables and eventually integrates with respect to the time variables. In this scenario, it becomes imperative to consider conditional expectations with respect to the spatial variables. Consequently, the entities that replace the time differences in this context are the spatial-related quantities:
	\begin{align*}
		(\Delta_{i}\bar{X})^2:=&(\bar{X}_{i} -\bar{X}_{{i-1}}\rho_{i}-(1-\rho_i)(2L-\bar{X}_{{i-1}}))^2,\ i\leq N_T,\\
		(\Delta_{N_T+1}\bar{X})^2:=&(L-\bar{X}_{{N_T}})^2.
	\end{align*} 
	
	The rationale behind introducing the aforementioned quantity is rooted in the fact that when the derivative operator is applied to  each term in \eqref{eq:PRF}, it will diminish the time order by a factor of $ 1 $. Consequently, to maintain stability and balance in the analysis, we introduce a stabilization mechanism that relies on spatial variables possessing an equivalent order given by $(\Delta_{i}\bar{X})^2$. To achieve this equilibrium, we employ the square of the distance between points.
	
	The subsequent result, which serves as the starting point of our analysis, stems directly from Theorem \ref{th:st1}.
	\begin{corol}Under the same conditions as in Theorem \ref{th:st1} and assuming that $ f $ does not depend on the space component, i.e. $ f(s,y)\equiv f(s) $, $ y\in\mathbb{R} $ with $ f(T)=0 $, we have 
		\begin{align}
			\E[f(\tau\wedge T)]  &= e^{\lambda T} \mathbb{E}\left[f(\zeta_{N_T}+\bar{\tau}^{\zeta_{N_T}}) \I_\seq{ {\bar{\tau}^{N_T}} \leq T- \zeta_{N_T}}
			\widehat{\theta}^{N_T}	\prod_{i=1}^{N_T}\theta^i \right],
			\label{eq:PRF1}\\
			\nonumber \mbox{ with } \quad
			\widehat{\theta} ^{N_T}& :=1+ \frac{a(L)-a_{N_T}}{a_{N_T}}.			\end{align}
	\end{corol}
	\subsection{Conditioning on the space variables}
	
	In this section, our objective is to state the conditional distribution of the jump times of the Poisson process, with respect to the spatial variables. In pursuit of this objective, we introduce the notations  $ \bar{X} :=\{\bar{X}_i;\ i\leq N_T\}$ and $ \rho:= \{\rho_i;\ i\leq N_T\}$. It is worth noting that, with probability one, $ \Delta_i\bar{X}\neq 0 $ for all $ i\leq N_T+1$. 
	 We will rely on this fact without reiterating it in the subsequent discussions. Additionally, it is pertinent to observe that the ``weight" random variable $\bar{\Gamma}_{N_T}\in \sigma(\bar{X}_i, T_i; i\leq N_T)$.

	In order to give proper definitions of how our differential operators will be defined, we first need to reconstruct the sequences of random variables $ \bar{X}_i $ and $ T_i$, $ i\in\mathbb{N} $, as follows.
	\begin{lem}
		\label{lem:8a}
		Let $ \tilde{Z}_{i} $ be a sequence of identically and independently distributed symmetric exponential random variables with parameter\footnote{That is $ \tilde{Z}_{i}=(-1)^{\eta_{i}} F_{i}$, where $ F_{i} \sim \mathrm{Exp}(1/\sqrt{2\lambda})$ and $ \eta_{i} \sim\mathrm{Ber}(1/2)$.} $1/\sqrt{2\lambda}$. 
		Inductively, define the following sequence of random variables:
		\begin{align*}
			\mathtt{Y}_{i}=\rho_{i} \mathtt{Y}_{i-1}+(1-\rho_{i})(2L- \mathtt{Y}_{i-1})+\sigma( \mathtt{Y}_{i-1})\tilde{Z}_{i}, \quad { \mathtt{Y}_0 = \bar{X}_0}
		\end{align*}
		Here $ \{\rho_i;i\in\mathbb{N}\} $ is assumed to be independent of $ \{\tilde{Z}_i;i\in\mathbb{N}\} $. Given $ \{ \mathtt{Y}_i;i\in\mathbb{N}\} $, define $ \tau_{i} $, $ i\in\mathbb{N} $, to be an independently distributed sequence of 
		random variables {such that $\tau_i$ follows a} 
		{Generalized Inverse Gaussian distribution\footnote{We denote by $GIG(a,b,p)$ the probabilistic distribution with density $f(x) = \frac{(a/b)^{p/2}}{2 K_p(\sqrt{ab})} x^{(p-1)} e^{-(ax + b/x)/2}$, $x>0$,
				where $a,b>0$ and $p\in\mathbb{R}$. Here $ K $ denotes the modified Bessel functions of the second kind . For more information, see the Appendix.} with parameters $a = 2\lambda$, $b= 2\lambda \mu^2_i$, $p = \frac{1}{2}$, denoted by $GIG(2\lambda,2\lambda {\mu_i^2}, \frac{1}{2})$, where with some abuse of notation, we let  $\mu _i= \frac{|\Delta \mathtt{Y}_i|}{{ \sigma_{i-1}  }\sqrt{2\lambda}}$} where $ \sigma_{i-1}=\sigma(\mathtt{Y}_{i-1}) $. Then, the following equality in law is satisfied $ \{(\rho_i,\bar{X}_i,\Delta T_i):i\in\mathbb{N} \}\stackrel{\mathcal{L}}{=}\{(\rho_i, \mathtt{Y}_i,\tau_i):i\in\mathbb{N}\}$ where $\Delta T_i = T_{i} - T_{i-1}$.
	\end{lem}
	\begin{proof}
		Let $ \rho $ be a Bernoulli random variable, independent of the exponential random variable $ T_1 $ and the Wiener process $ W $ which is also independent of $ T_1 $. Note that the joint density of $ (\rho x+(1-\rho)(2L-x) + \sigma_0 W_{T_1},T_1) $ conditioned to $ \rho $ is given by  
		\begin{align}
			\frac{\lambda}{\sqrt{2\pi\sigma_0^2 s}}\exp\left(-\frac{(y-\rho x-(1-\rho)(2L-x))^2}{2\sigma_0^2s}-\lambda s\right)1_{\{s>0\}}.
			\label{eq:func}
		\end{align}
		
		First, we compute the marginal density of $ \rho x+(1-\rho)(2L-x)+\sigma_0 W_{T_1} $.
		This is done using characteristic functions as follows.
		Using that $T_1$ is exponentially distributed and independent from $W$ we have
		\begin{align*}
			\mathbb{E}[e^{\theta W_{T_1}}] = \mathbb{E}[\mathbb{E}[e^{\theta W_{s}}|T_1]\big|_{s= T_1}]= \mathbb{E}[e^{\frac{1}{2}\theta^2 T_1}] = \int^\infty_0 {\lambda} e^{(\frac{1}{2}\theta^2 -\lambda)s} ds = \frac{\lambda}{\lambda-\frac{1}{2}\theta^2 }.
		\end{align*}
		{This corresponds to the Laplace transform of symmetric exponential distribution with parameter $ 1/\sqrt{2\lambda} $.
			From direct computation, the density of $\tilde Z_1$ is given by
			\begin{align*}
				f_{{W_{T_1}}}(y) = f_{{\tilde Z_1}}(y) = \frac{1}{2}\sqrt{2\lambda} e^{-\sqrt{2\lambda}{|y|}}, \quad y \in \mathbb{R}.
			\end{align*}
			Then, the conditional density function of $ \mathtt{Y}_1$ given $\rho$ is given by
			\begin{align*}
				f_{ \mathtt{Y}_{1}}(y)  = \frac{1}{2\sigma_0} \sqrt{2\lambda} e^{-\frac{|y - \rho x - (1-\rho)(2L-x)|}{\sigma_0}\sqrt{2\lambda}} = \frac{1}{2\sigma_0}\sqrt{2\lambda}e^{-2\lambda\mu_1 }, \quad y \in \mathbb{R},
			\end{align*}
			with $\mu_1 = \frac{|y - \rho x - (1-\rho)(2L-x)|}{\sigma_0{\sqrt{2\lambda}}}$. To obtain the conditional density of $\tau_1$ given $ \mathtt{Y}_1$ and $\rho$, we see that 
			\begin{align*}
				f_{\tau_1 |  \mathtt{Y}_1}(s|y) & = \frac{\lambda}{\sqrt{2\pi\sigma_0^2 s}}\exp\left(-\frac{(y-\rho x-(1-\rho)(2L-x))^2}{2\sigma_0^2s}-\lambda s\right)1_{\{s>0\}} 2\sigma_0\frac{1}{\sqrt{2\lambda}} e^{2\lambda\mu_1 }\\
				& = \frac{\sqrt{\lambda}}{\sqrt{\pi s}}\exp\left(-\lambda \mu_1^2 s^{-1}-\lambda s\right)1_{\{s>0\}}\exp\left(2\lambda\mu_1\right).
			\end{align*}
			By using the fact that $K_{\frac{1}{2}}(z) = \sqrt{\frac{\pi}{2}}\frac{e^{-z}}{\sqrt{z}}$, where $K_\nu(z)$ is a Bessel function of the second kind, we can deduce that the above is a Generalized Inverse Gaussian density with parameters $( a,b,p)$ where $a = 2\lambda$, $b= 2\lambda \mu_1^2$, $p = \frac{1}{2}$.}
		In order to obtain the result, it is enough to do it by partitioning the expectation using $ 1_{\{T_1<...<T_n<T<T_{n+1}\}} $
		which leads to the announced result. Actually, given the above result, it is not difficult to obtain
		\begin{align*}		&\mathbb{E}\left[f(\bar{X}_{i},\Delta T_i)1_{\{N_T\geq i\}}\,\Big|\,\bar{X}_{i-1},\rho_{i-1}, \Delta T_{j},\,\,j=1\dots i-1\right] \\
			& = \mathbb{E}\left[f\big(\mathtt{Y}_i,\tau_i\big)1_{\{N^\tau_T\geq i\}}\,\Big|\, \mathtt{Y}_{i-1}=\bar{X}_{i-1},\rho_{i-1}, \tau_j = \Delta T_{j}, \,j=1\dots i-1\right].
		\end{align*}
		Here $ N^\tau $ denotes the renewal process generated from the times between jumps given by $ \{\tau_i;i\in\mathbb{N}\} $.
	\end{proof}
	
	
	On the set $ \{N^\tau_T=n\} $ and for the last interval $[ \zeta_n,\bar{\tau}^n ]$, we remark that conditioning on the space variables does not imply any change in the law of $ \bar{\tau}^n $ except that its parameter becomes $ (\Delta_{n+1} \mathtt{Y})^2:=(L-\mathtt{Y}_n)^2$.  That is, the law of $ \bar{\tau}^n $ is a L\'evy distribution with parameter $|\Delta_{n+1} \mathtt{Y}| / \sigma_n>0$. In other words, the probability density of $\bar \tau^n$ is given by $ p_{\bar{\tau}^n}(s)=\frac{|\Delta_{n+1} \mathtt{Y}|}{\sqrt{a_n2\pi s^3}}\exp(-\frac{(\Delta_{n+1} \mathtt{Y})^2}{2a_ns}) $, $ s>0 $ with $ p_{\bar{\tau}^n}(0)=0$ by continuity. Here, we abuse the notation slightly by letting $ \zeta_j= (\sum_{i=1}^j\tau_i) \wedge T$. A similar comment applies to $ \sigma_i $, $ \mathcal{I}_i $ and $ \mathbb{E}_{M,k,n}\left[\cdot\right] $ which we define for a positive random variable $M$ being measurable with respect to $\sigma( \mathtt{Y},\zeta,\rho)$ as 
	\begin{align*}
		\E_{M,k,n}[X]:=\E[XM^{-1}\I_\seq{N^\tau_T=n} \,|\, \mathcal{F}_{\zeta_k}, \mathtt{Y},\rho].
	\end{align*}
	Here, $ \mathcal{F}_{\zeta_k}:=\sigma(\zeta_i;i\leq k) $. 
	In the first part, we will use $ M=1 $. Later, $ M $ will denote the Malliavin variance whose inverse moment properties are discussed in Section \ref{sec:Mv}.

	This completes our set-up for the consideration of differential calculus on this space. 
	Using this new set-up, one clearly has that the following basic IBP formula is satisfied for $ n\in\mathbb{N}_0 $ and $ f\in\mathscr{C}^1_b(\mathbb{R}) $ with $ f(T)=0 $:
	\begin{align}
		\label{eq:Last}
		\E_{1,n,n}[f'(\bar{\tau}^n)\I_\seq{\zeta_n+\bar{\tau}^n<T}]=&\E_{1,n,n}[f(\bar{\tau}^n) \widehat{{\mathcal{I}}}_{n+1}(1)\I_\seq{\zeta_n+\bar{\tau}^n<T}],\\
		\widehat{{\mathcal{I}}}_{n+1}(G):=&G\left\{\frac{3}{2\bar{\tau}^n
		}-\frac{(\Delta_{n+1} \mathtt{Y})^2}{2a_n(\bar{\tau}^n)^2} \right\}-{D}_{\bar{\tau}^n}G.
		\nonumber
	\end{align}
	In the above formula, conditioned on $  \mathtt{Y} $ and $ \rho $, we assume that $ G=G(\bar{\tau}^n) $ for $ G\in\mathscr{C}^\infty_p(\mathbb{R}) $ so that $ D_{\bar{\tau}^n}G=G'(\bar{\tau}^n) $.
	We also remark that the result is applicable if $ f $ or $ G $ depend on $\zeta_1, \cdots, \zeta_n$ or $\bar{X}  $. 
	
	Due to Lemma \ref{lem:8} in the Appendix, we have that for a positive non-random constant $ C>0 $
	\begin{align}
		\label{eq:estTn}
		\E_{1,n,n}\left [ \left|\widehat{{\mathcal{I}}}_{n+1}(G)\right|\I_\seq{\zeta_n+\bar{\tau}^n<T}\right ]\leq C\frac{\sup_{t\in [0,T]}(|G(t)|+|G'(t)|)}{(\Delta_{n+1} \mathtt{Y})^2}.
	\end{align}
	It is worth noting that the inequality described above may be referred to as a spatial degeneration inequality, mirroring the circumstances encountered when examining IBP formulas in the context of stopped processes, where analogous estimates were denoted as time degeneration inequalities.

	In order to give the IBP formula, we need to introduce the following operators.
	We define\footnote{One may define the equivalent spaces $ {\mathbb{S}} $ as in the case of stopped diffusions but we refrain to do it in order to simplify the presentation.} for a function $ G\in \mathscr{C}^1_p(\mathbb{R}) $, the r.v. $ G\equiv G(\tau_i) $.
	We define the operators to $ ({\mathtt{I}}_i,{\mathtt{D}}_i)$ for $ i\leq n $ on the set $ \{N^\tau_T=n\} $. That is, for $ G\equiv G(\tau_i) $,
	\begin{align*}
		{{\mathtt{I}}}_{i}(G):=&G\left(\lambda+(2\tau_i)^{-1}+\frac{(\Delta_i \mathtt{Y})^2}{2a_{i-1}\tau_i^2}\right)-\mathtt{D}_{i} G\\
		\mathtt{D}_{{i}}G:=&G'(\tau_i).
	\end{align*}
	\begin{remark}
		In the current framework, it is crucial to observe that, based on the construction outlined at the outset of this section, the random variables $ \Delta_i \mathtt{Y} $ exhibit no dependence on $ \tau_i $. Consequently, it follows that $ \mathtt{D}_i \Delta_j \mathtt{Y} = 0 $ for any $ j\in\mathbb{N} $.	\end{remark}
	We also have from Lemma \ref{lem:8}, the space degeneration estimates for a non-random constant $ C>0 $ and $ k\in\{0,1,2\} $:
	\begin{align}
		\label{eq:estk}
		\E_{1,i-1,n}\left[|{\mathtt{I}}_i(\tau_i^{-k})|
		\right]\leq &C\left({1+|\Delta_i\mathtt{Y}|^2+|\Delta_{i} \mathtt{Y}|^{-(2k+1)}}\right) 
	\end{align}
	
	It's important to note that in many situations, we may employ a function $ G(\tau_1,\cdots,\tau_n) $ defined on the event $ {N^\tau_T=n} $. In such cases, the definitions provided above are to be understood conditionally to all times between jumps except $ \tau_i $, besides the conditioning involving spatial variables and all $ \rho_i $ for $ i\in\mathbb{N} $.
	
	Furthermore, we will continue to use $ \delta_t(\tau_i) $ to denote the Dirac delta distribution function at the point $ t $.
	
	Similar to the approach taken in \cite{frikha:kohatsu:li}, we anticipate that employing the symbol $ {\mathtt{I}}_i $ for the dual operator of the derivative operator $ {\mathtt{D}}_i$ will not result in confusion, as the context will consistently clarify the object under discussion. Additionally, the formulas for $ {\mathcal{I}}_i $ and {$\mathcal{D}_i$} which are used from now on are restated using $ \mathtt{Y} $ instead of $ \bar{X} $. That is, we adopt the notation convention of using the same symbol, despite the possibility of being expressed with $\mathtt{Y}$.
	\begin{lem}
		\label{lem:10a}
		With the above definitions, we have the following IBP formula for $ 0\leq t_1<t_2 $ and $ f, G\in \mathscr{C}^1_p (\mathbb{R})$:
		\begin{align}
			\nonumber
			\E_{1,i-1,n}	\left[f'({\tau_i})G({\tau_i})\I_\seq{t_1<\tau_i<t_2}
			\right]=&\E_{1,i-1,n}
			\left[f({\tau_i}){{\mathtt{I}}}_{i}(G)\I_\seq{t_1<\tau_i<t_2}
			\right]\\
			&+\E_{1,i-1,n}\left[f({\tau_i})G(\tau_i)(\delta_{t_2}({\tau_i})-\delta_{t_1}({\tau_i}))\right].
			\label{eq:timeIBP}
		\end{align}
		In the particular case that $ t_1=0 $, using that $  \mathtt{Y}_{i-1}\neq  \mathtt{Y}_i $ a.s., the above formula reduces to 
		\begin{align}
			\E_{1,i-1,n}	\left[f'({\tau_i})G({\tau_i}) \I_\seq{0<\tau_i<t_2}\right]=&\E_{1,i-1,n}\left[f({\tau_i}){{\mathtt{I}}}_{i}(G)\I_\seq{0<\tau_i<t_2}
			\right]+\E_{1,i-1,n}\left[f({\tau_i})G(\tau_i)\delta_{t_2}({\tau_i})
			\right].
			\label{eq:timeIBP0}
		\end{align}
	\end{lem}
	\begin{proof}
		The proof follows using the usual IBP formula for Lebesgue integrals.	
	\end{proof}

	The next result states the independence of the weight $ \theta^j $ with respect to the noise in $ \tau^i $ once they are rewritten using  the set-up stated at the beginning of this section. This result simplifies calculations greatly.

	\begin{lem}
		\label{lem:4}
		$ {\mathtt{D}}_i{\theta}^j=0 $ for $ i\neq j $.
	\end{lem}
	\begin{proof}
		Note that as $ \tilde{Z}_j=\sigma_{j-1}^{-1}\Delta_{j} \mathtt{Y} $, one obtains that $ \theta^{j} $ is only a function of  $ ( \mathtt{Y},\rho,\tau_j)$ by observing \eqref{eq:ab}. Using that $ j\neq i $, we get the result.	
	\end{proof}

	Building upon the information provided in the preceding proof and in conjunction with \eqref{eq:estk} and Lemma \ref{lem:8}, we can derive the following space degeneration estimates:
	\begin{equation}
		\label{eq:estthetai}
		\begin{aligned}
			\E_{1,i-1,n}\left[|\theta^i|
		\right]\leq &C\left({1+|\Delta_{i} \mathtt{Y}|^3}\right),\\
		\E_{1,i-1,n}\left[|{\mathtt{I}}_i(\theta^i)|
		\right]\leq &C\left({1+|\Delta_{i} \mathtt{Y}|^{-2}}+|\Delta_i\mathtt{Y}|^5\right).
		\end{aligned}
	\end{equation}

	In fact, using \eqref{eq:ab} and \eqref{eq:I1} twice, we see that by linearity it is enough to estimate and upper bound $ \E_{1,i-1,n}\big[|{\mathtt{I}}_i({{\mathcal{I}}}^{j}_{i}(c_j^i))|
	\big] $ for $ j=1,2 $. Now we apply \eqref{eq:exta}, \eqref{eq:link:integral:hermite:pol} and the linearity of $ {\mathtt{I}}_i(\cdot) $ to obtain
	\begin{align*}
		\E_{1,i-1,n}\left[|{\mathtt{I}}_i({{\mathcal{I}}}^{j}_{i}(c^i_j))|
		\right]\leq \sum_{k=0}^j\E_{1,i-1,n}\left[|{\mathtt{I}}_i\left(\mathcal{H}_{j-k}(a_{i-1}\tau_i,\sigma(\mathtt{Y}_{i-1})\tilde{Z}_i)\right)\mathcal{D}_i^kc^i_j|
		\right].
	\end{align*}
Using explicit expressions for Hermite polynomials, the definition of $ \mathtt{I}_i $ and \eqref{eq:estk}, we obtain the second estimate in \eqref{eq:estthetai}.
	
	\subsection{The transfer formula}
	The subsequent equalities are all stated with respect to conditional expectations involving the spatial variables. In particular, the symbol $\stackrel{\E_{M,k,n}}{=}$ denotes equality within the respective conditional expectation context.
	
	With this notation established, we can express the formula for transferring derivatives in the final interval as follows:
	\begin{lem} Let $ f\in \mathscr{C}^1_p(\mathbb{R}) $ with $ f(T)=0 $ then, with $ \overleftarrow{\theta}_{n+1}^e=\hat{\theta}^n$, it holds 
		\begin{align}
			\label{eq:lastIBPst}
			\E_{M,n,n}[f'(\zeta_n+\bar{\tau}^n)\I_\seq{\zeta_n+\bar{\tau}^n<T}\hat{\theta}^n]=\partial_{\zeta_n}
			\E_{M,n,n}[f(\zeta_n+\bar{\tau}^n)\I_\seq{\zeta_n+\bar{\tau}^n<T}\overleftarrow{\theta}_{n+1}^e].
		\end{align}
		Also, we have that, on the set $ \{\zeta_n=T \}$, 	$ \E_{M,n,n}[f(\zeta_n+\bar{\tau}^n)\I_\seq{\zeta_n+\bar{\tau}^n<T}\hat{\theta}^n]=0 $.
	\end{lem}
	\begin{proof} Given the conditioning we may assume without loss of generality that $ \hat{\theta}^n=1 $.
		Recall that the IBP formula for the last interval is obtained using  a L\'evy distribution as stated in \eqref{eq:Last}. Hence, by employing Leibniz's rule for interchanging derivatives and integrals, it can be deduced that the right-hand side of \eqref{eq:lastIBPst} is differentiable with respect to $ \zeta_n $, and:
		\begin{align*}
			&f'(\zeta_{n}+ {\bar{\tau}^{n}})
			\I_\seq{ {\bar{\tau}^{n}} < T- \zeta_{n}}
			\stackrel{\E_{M,n,n}}{=}f(\zeta_{n}+{\bar{\tau}^{n}})
			\I_\seq{ {\bar{\tau}^{n}} < T- \zeta_{n}}	\widehat{{\mathcal{I}}}_{n+1}(1)	.
		\end{align*}
		For the second result, it is enough to note that $ \E_{M,n,n}[f(\zeta_n+\bar{\tau}^n)\I_\seq{\zeta_n+\bar{\tau}^n<T}] $ is a continuous function of $ \zeta_n $.
	\end{proof}
	We would like to mention that the notation $\overleftarrow{\theta}_{n+1}^e$ is adopted from Section 5.1 of \cite{frikha:kohatsu:li}. We have retained this notation to emphasize its significance within the tree structure, as explained in \cite{frikha:kohatsu:li}.
	
	Now, let's present the general transfer formula for the remaining intervals.
	

	\begin{lem}
		\label{lem:6}
		Let $ f\in \mathscr{C}^1_b(\mathbb{R}_0) $ with $ f(T)=0 $. It holds 
		\begin{align}
			\E_{M,i,n}[f'(\zeta_{i+1})\I_\seq{\zeta_{i+1}<T}\theta^{i+1}]=&
			\partial_{\tau_i}\E_{M,i,n}[f(\zeta_{i+1})\I_\seq{\zeta_{i+1}<T}\overleftarrow{\theta}^e_{i+1}]
			\label{eq:4.1}.
		\end{align}
		Here, $ \overleftarrow{\theta}^e_{i+1}:=\theta^{i+1} $. 
		Moreover, the following boundary condition is satisfied
		$$ \E_{M,i,n}[f(\zeta_{i+1})\I_\seq{\zeta_{i+1}<T}\overleftarrow{\theta}^e_{i+1}]\big|_{\zeta_{i}=T} =0,$$  as well as the following time degeneracy inequality\footnote{Recall that the norm $ \|\cdot\|_{0,p} $ is a norm on the Wiener space conditional to $ \rho $ and $ N^\tau $.}:
		\begin{align*}
			\|\overleftarrow{\theta}^e_{i+1}\|_{0,p}\leq C(\zeta_{i+1}-\zeta_i)^{-1/2}.
		\end{align*}
	\end{lem}
	\begin{proof}
		The proof can be established through straightforward differentiation of the right-hand side of equation \eqref{eq:4.1}. In particular, note that $ \theta^{i+1} $ does not depend on $ \tau_i $. Concerning the boundary conditions, it's worth noting that we have $f(T) = 0$. Additionally, the time degeneracy estimate has been previously established in \cite{frikha:kohatsu:li}.
	\end{proof}
	
	\subsection{The IBP formula}
	The concepts applied here bear a resemblance to those used in the case of stopped processes, albeit with a shift in the roles of time and space.
	
	First, let's focus on the last interval. Without delving into technical details, we can derive the IBP formula on the set $\left\{{N_T=n}\right\}$ from Lemma \ref{lem:10a}:
	\begin{align*}
		&(\Delta_{n+1} \mathtt{Y})^2f'(\zeta_{n}+ {\bar{\tau}^{n}})
		\I_\seq{ {\bar{\tau}^{n}} \leq T- \zeta_{n}}\hat{\theta}^n
		\stackrel{\E_{M,n,n}}{=}(\Delta_{n+1} \mathtt{Y})^2
		f(\zeta_{n}+{\bar{\tau}^{n}})
		\I_\seq{ {\bar{\tau}^{n}} \leq T- \zeta_{n}}	\hat{\theta}^n\widehat{{\mathcal{I}}}_{n+1}(1)	.
	\end{align*}
	Furthermore, the above expression satisfies the following space degeneration estimate due to  \eqref{eq:estTn}: 
	\begin{align*}
		\left|
		(\Delta_{n+1} \mathtt{Y})^2
		f(\zeta_{n}+{\bar{\tau}^{n}})\hat{\theta}^n
		\widehat{{\mathcal{I}}}_{n+1}(1)\right|\I_\seq{\zeta_n+\bar{\tau}^n<T}\stackrel{\E_{M,n,n}}{\leq }   {C}\sup_{t\in[0,T]}|f(t)|.
	\end{align*}
	
	Now, to make use of {the transfer formula} provided in Lemma \ref{lem:6} for all time intervals leading up to the interval where the IBP formula will be applied, we define, for $k\geq i$
	\begin{align*}
		G_k(\zeta_k)\equiv G_k(n,\zeta_k, \mathtt{Y},\rho):=\E_{M,k,n}\left[f(\zeta_n+\bar{\tau}^n)\I_\seq{\zeta_n+\bar{\tau}^n<T}\hat{\theta}
		^n\prod_{j=k+1}^{n}\theta^j\right].
	\end{align*}
	With the definition provided above and using \eqref{eq:estthetai}, we can establish that $G_k(T)=0$, $|G_k(\zeta_k)|\leq C \prod_{j=k+1}^{n+1}(1+|\Delta_j \mathtt{Y}|^3)$. Consequently, the application of the transfer formulas \eqref{eq:lastIBPst} and \eqref{eq:4.1}, along with Lemma \ref{lem:10a}, yields
	\begin{align*}
		(\Delta_{i} \mathtt{Y})^2f'(\zeta_n+\bar{\tau}^n
		)
		\I_\seq{ 
			\zeta_n+\bar{\tau}^n<T
		}\prod_{j=i}^{n}\theta^j
		\stackrel{\E_{M,i-1,n}}{=}&
		(\Delta_{i} \mathtt{Y})^2G'_i(\zeta_i)\I_\seq{ \zeta_i \leq T}
		\theta^i\\
		\stackrel{\E_{M,i-1,n}}{=}&
		(\Delta_{i} \mathtt{Y})^2G_i(\zeta_i)\I_\seq{ \zeta_i \leq T}{{\mathtt{I}}}_i(\theta^i).
	\end{align*}
	Again from \eqref{eq:estthetai}, the space degeneration estimate in this case is 
	\begin{align*}
		\left|(\Delta_{i} \mathtt{Y})^2
		f(\zeta_{i})\right|
		\I_\seq{ \zeta_i \leq T}\hat{\theta}^n\prod_{j=i+1}^{n} |\theta^j|	|{{\mathtt{I}}}_i(\theta^i)|
		\stackrel{\E_{M,i-1,n}}{\leq }
		{C}\sup_{t\in[0,T]}|f(t)|
		\prod_{j=i+1}^{n}
		(1+|\Delta_j \mathtt{Y}|^3)
		(1+|\Delta_i \mathtt{Y}|^7).
	\end{align*}

	We can now leverage the fact that, as indicated by Lemma \ref{lem:8a}, ${|\Delta_i \mathtt{Y}|= \sigma(\mathtt{Y}_{i-1})|\tilde Z_i|,i=1,...,N^\tau_T}$, where $\tilde Z_i$ represents an i.i.d. sequence of symmetric exponential random variables characterized by a parameter of $1/\sqrt{2\lambda}$.
	Hence, for any integer $n$, 
	\begin{align*}
		&\E\left[\left|
		f(\zeta_{n}+{\bar{\tau}^n})\right|
		\I_\seq{ \zeta_n+\bar{\tau}^n\leq T}|\hat{\theta}^n|\left(\sum_{i=1}^{{ n }}(\Delta_{i} \mathtt{Y})^2\prod_{j=i+1}^{n}|\theta^j|	|{{\mathtt{I}}}_i(\theta^i)|\prod_{k=1}^{i-1}|\theta^k|+(\Delta_{n+1} \mathtt{Y})^2
		|\widehat{{\mathcal{I}}}_{n+1}(1)|\prod_{k=1}^{n}|\theta^k|
		\right)M^{-1}\I_\seq{N^\tau_T=n}\right]
		\\
		&\quad {\leq }
		(n+1)C^n\sup_{t\in[0,T]}|f(t)|\P(N_T=n).
	\end{align*}
	Here, $ M:=	\sum_{{i}=1}^{N^\tau_T+1}\left(\Delta_{i} \mathtt{Y}\right)^2 $. We remark here that once all the integration by parts have been performed, expressions for the weights can be written using   the original formulation in \eqref{sec:pre} due to Lemma \ref{lem:8a}.
	
	As it was done in \cite{frikha:kohatsu:li}, we may give an algebraic structure to the  the IBP formula. We therefore refer the reader to that reference for the notation and terminology that follows. 
 
 In this case, the structure is easier and it follows the same graph as in the stopped case except that all terms with arrows corresponding to $ (\overleftarrow{\theta}^c,\overleftarrow{\theta}^\partial) $ do not appear in this case.\footnote{We refer the reader for details on this ``path algebraic structure'' to Section 5.1 in  \cite{frikha:kohatsu:li}}
	%
	%
	%
	%
	%
	%
	In conclusion, from the above calculations one understands that the algebraic structure in this case is formed only of components $ I_{i}^{n+1}=(0,\ldots,0,I,e,\ldots,e) $, $ {i}\leq N_T+1 $. These components have exchange r.v.'s in all intervals $ [\zeta_k,\zeta_{k+1}] $ for $ k\in\{i,...,N_T\} $, denoted by $ e $, an integration by parts on the interval $[\zeta_{i-1},\zeta_{i}]   $, denoted by $ I $ and no change in weight r.v.'s in all previous intervals $ [\zeta_k,\zeta_{k+1}] $ for $ k\in\{0,...,i-2\} $ which is denoted by $ 0 $ in the component $ I_i^{n+1} $.

	Putting all the above reasoning together, one obtains the IBP formula for the exit time. For this, we need to define
	for $ i\leq n+1 $ on the set $ \{N_T=n\} $		\begin{align*}
		\theta^{{I}_{i}}:=&\hat{\theta}^n
		\prod_{j={i}+1}^{n}
		\overleftarrow{\theta}_j^e\times
		{\mathtt{I}}_{i}(
		{\theta}^{{i}})\times
		\prod_{l=1}^{{i}-1}{\theta}^l= \hat{\theta}^n
		\prod_{j={i}+1}^{n}
		\theta^j\times
		{\mathtt{I}}_{i}(
		{\theta}^{{i}})\times
		\prod_{l=1}^{{i}-1}{\theta}^l\\
		\theta^{{I}_{n+1}}:=&
		\hat{\theta}^n\hat
		{\mathcal{I}}_{n+1}(
		1)\times
		\prod_{l=1}^{n}{\theta}^l.
	\end{align*}
	
	\begin{theorem}
		\label{th:8}
		Let $x>L$ and $ f\in \mathscr{C}^1_b(\mathbb{R}) $ such that $ f(T)=0 $. Then, it holds 
		\begin{align*}
			\E\left[f'(\tau\wedge T)\right]=&e^{\lambda T}
			\E\left[M^{-1}\sum_{i=1}^{N_T+1}	(\Delta_{i}\bar{X})^2 f(\zeta_{N_T}+ {\bar{\tau}^{N_T}})\I_\seq{ {\bar{\tau}^{N_T}} \leq T- \zeta_{N_T}}
			\theta^{I_i}
			\right].
		\end{align*}
		The random variable which appear within the expectation on the right-hand side above is integrable.
	\end{theorem}
	The proof of the above result is similar to the case of stopped processes given in Theorem 5.1 of \cite{frikha:kohatsu:li}. Algebraically, the proof essentially follows from applying the recursive relation given in (5.2), the transfer formula and the iteration step given in (5.4) of Theorem 5.1 \cite{frikha:kohatsu:li}. The structure given here is much simpler compared to \cite{frikha:kohatsu:li} since all boundary terms are zero. Furthermore, the necessary inequalities and estimates are provided at the beginning of this section. It is important to note that integrability can be inferred from the space degeneration estimates outlined in Lemma \ref{lem:8}. Lastly, the assumption $f(T)=0$ can be readily eliminated by introducing the function $g(t) = f(t) - f(T)$.

	\section{Appendix}
	\subsection{The Malliavin variance}
	\label{sec:Mv}
	
	Define the Malliavin variance as
	\begin{align*}
		M:=	\sum_{{i}=1}^{N_T+1}\left(\Delta_{i}\bar{X}\right)^{2}.
	\end{align*}	
	
	We remark here that an equivalent definition using $(\rho, \mathtt{Y},\tau)$, instead of $(\rho,\bar{X},T)$ is possible and the study of the inverse moments is equivalent due to Lemma \ref{lem:8a} which states the equality in law. We prefer this version as we will be using geometrical properties of the path of $\bar{X}$.
	The derivation of the IBP formula necessitates the existence of inverse moments of the Malliavin variance on the set ${\bar{\tau}^{N_T} \leq T - \zeta_{N_T}}$. The reason for the inverse integrability is that each space increment $\Delta_{i}\bar{X}$ has a positive probability of being close to $\frac{L - x}{N_T}$. This provides a lower bound for the variable $M$, ensuring that its inverse is integrable in any $L^p(\mathbb{P})$ norm. This result is established in the following section.
	\begin{lem}
		For any $ q>0 $, $ k\in\mathbb{N} $ and $ x>L $, we have that 
		$$ \E\left[\left(\sum_{{i}=1}^{N_T+1}\left(\Delta_{i}\bar{X}\right)^{2k}\right)^{-q} \I_\seq{\zeta_{N_T} + \bar{\tau}^{N_T}<T } \prod_{i=1}^{N_T+1}\I_\seq{\bar{X}_i>L}\right]<\infty. 
		$$
	\end{lem}
	\begin{proof}
		One  considers the set  \begin{align*}
			A:=\left \{\exists {i}\in\{1,...,N_T+1\};\ \left|\Delta_{i}\bar{X}
			\right| \geq \frac{L-x}{N_T +1}\right \}.
		\end{align*} 
		Note that {$ \mathbb{P}(A^c, \bar{X}_i > L, i=1,\cdots, N_T)=0 $} because the following inequality is satisfied 
		\begin{align*}
			\sum_{{i}=1}^{N_T+1}\left |\Delta_{i}\bar{X}\right| \geq L-x,
		\end{align*}
		\noindent on the set $\left\{\bar{X}_i > L, i=1, \cdots, N_T\right\}$. The proof of the above inequality can be explained easily in geometrical terms.  In the case that $ \rho_{i}=1 $ then $ |\Delta_{i}\bar{X}
		| $ stands for the distance between the two points $ \bar{X}_{i} $ and $ \bar{X}_{{i}-1} $. In the case that $ \rho_{i}=0 $ then the value $ \bar{X}_{{i}-1} $ is reflected with respect to the line $ y=L $ but it also has to finish at a point $  \bar{X}_{i} >L$ and therefore $ | \bar{X}_{i}-(2L-\bar{X}_{{i}-1})| $ is larger than $ |\bar{X}_{i}-\bar{X}_{{i}-1}| $ . As the initial point and final point of the process $ \bar{X} $ are fixed at $ x $ and $ L $ respectively, the above inequality follows.
		
		%
		%
		Then, using Chebyshev's inequality, one obtains:
		\begin{align*}
			\E\left[\left(\sum_{{i}=1}^{N_T+1}\left(\Delta_{i}\bar{X}\right)^{2k}\right)^{-q}  \I_\seq{\zeta_{N_T} + \bar{\tau}^{N_T}<T } \prod_{i=1}^{N_T+1} \I_\seq{\bar{X}_i>L}\right]\leq  \E\left[\frac{(N_T+1)^{2qk+1}}{(L-x)^{2qk}}\I_A\right]<\infty.
		\end{align*}
	\end{proof}
	\subsection{Auxiliary lemmas for L\'evy and generalized inverse Gaussian distributions}
	\
	\begin{lem}
		\label{lem:10}Let $ f:\mathbb{R}_+\to\mathbb{R} $ be a measurable and bounded function. Then, it holds 
		\begin{align*}
			\E\left[(2\rho_1-1)\mathcal{I}_{1}(1)\delta_L(\bar{X}_{1})f(\zeta_1)\,|\,N_T=1\right]=-T^{-1}\E[f(\bar{\tau}^x)\I_\seq{\bar{\tau}^x\leq T}].
		\end{align*}
	\end{lem}
	\begin{proof}
		The proof consists of explicit calculations. In fact, 
		\begin{align*}
			\E\left[(2\rho_1-1)\mathcal{I}_{1}(1)\delta_L(\bar{X}_{1})f(\zeta_1)\,|\,N_T=1\right]&=T^{-1}\int_0^T \frac{L-x}{\sigma(x)s}g_s(x,L)f(s) \, \d s.
		\end{align*}
		The above equals $-T^{-1}\E[f(\bar{\tau}^x)\I_\seq{\bar{\tau}^x\leq T}]  $ because the density function of $ \bar{\tau}^x $ is explicitly known and given by 
		$ \frac{x-L}{\sigma(x)s}g_s(x,L) $	.
	\end{proof}
	\begin{lem}
		\label{lem:8}For $ k\in\mathbb{N} $ and $ 1\leq i\leq n\in\mathbb{N} $,
		\begin{align*}
			\E_{1,n,n}[(\bar{\tau}^n)^{-k}\I_\seq{\zeta_n+\bar{\tau}^n\leq T}] & = \E[(\bar{\tau}^n)^{-k}\I_\seq{\zeta_n+\bar{\tau}^n\leq T}\I_\seq{ N_T=n}\,|\,\mathcal{F}_{\zeta_n}, \mathtt{Y},\rho]\leq 
			C(\Delta_{n+1} \mathtt{Y})^{-2k} ,\\
			\E_{1.i-1,n}\left[\tau_i^{-k} \right] &  = \E\left[{\tau_i}^{-k} \I_\seq{ N_T=n}\,|\,\mathcal{F}_{\zeta_{i-1}}, \mathtt{Y},\rho\right]\leq  {C\left(1+ {|\Delta_i \mathtt{Y}|^{1-2k}}\right)}.
		\end{align*}
	\end{lem}
	The proof of the above statement is done using explicit integral expressions of inverse moments related to the L\'evy distribution and generalized inverse Gaussian distributions on $ \mathbb{R}_+ $. In fact, the first inequality follows using the explicit density of $ \bar{\tau}^n $ together with the inequality $x^{2k} e^{-cx^2} \leq \left(\frac{k}{2 c}\right)^{k} e^{-k}$. 
	Therefore, applying this inequality to
	\begin{align*}
		\E[(\bar{\tau}^n)^{-k}\I_\seq{\zeta_n+\bar{\tau}^n\leq T}
		\,|\,{N_T 
			= n }, \mathcal{F}_{\zeta_n}, \mathtt{Y},\rho]=\int_0^{T-\zeta_n}s^{-k}\frac{|\Delta_{n+1}  \mathtt{Y}|}{\sqrt{2\pi}(a_ns)^{3/2}}\exp\left(-\frac{(\Delta_{n+1}  \mathtt{Y})^2}{2a_ns}\right) \, \d s,
	\end{align*}
	one obtains the result in the first case using the inequality for $ c=(4a_ns)^{-1} $.

	The second inequality follows using the explicit form of the  $\mathrm{GIG}$ distribution. The expectation of $ \tau_i^{-k} $ can be reinterpreted using the GIG distribution with parameters $p = -k+\frac{1}{2}$, $a = 2\lambda$, $b = 2\lambda \mu_i^2$ or, in other words, the GIG$( 2\lambda, 2\lambda \mu_i^2,-k+\frac{1}{2})$ distribution.  This introduces a change of normalization constant which gives the inequality.  When doing these calculations, one uses the following facts about modified Bessel functions of the second kind $ K_{3/2}(x)=(1+x^{-1})K_{1/2}(x)= (1+x^{-1})\sqrt{\frac{\pi}{2x}}e^{-x}$. In general, one may write for
	$ \nu-\frac 12\in\mathbb{Z} $, (see, for example, page 925 in \cite{gradshteyn2007})
	\begin{align*}
		K_{\nu}(x)=K_{|\nu|}(x)=K_{1/2}(x)\sum_{j=0}^{\lfloor{|\nu|-\frac 12}\rfloor}\frac{(j+|\nu|-\frac 12)!}{j!(-j+|\nu|-\frac 12)!}(2x)^{-j}.
	\end{align*}
	In fact, the GIG density multiplied by $x^{-k}$ is given by
	\begin{align*}
		& \frac{1}{x^k} \frac{(a/b)^{p/2}}{K_p(\sqrt{ab})} x^{p-1} e^{-(ax+ b/x)/2}\\
		& = \frac{(a/b)^{k/2}K_{p-k}(\sqrt{ab})}{K_p(\sqrt{ab})}\left[\frac{(a/b)^{(p-k)/2}}{K_{p-k}(\sqrt{ab})} x^{p-k-1} e^{-(ax+ b/x)/2}\right].
	\end{align*}
	Here $\sqrt{ab} = 2\lambda \mu_i$ and $a/b = 1/\mu^2_i$. Hence, recalling that $p = \frac{1}{2}$, the previous quantity is upper-bounded by
	\begin{align*}
		\mu_i^{-k} \sum_{j=0}^{\lfloor{|\nu|-\frac 12}\rfloor}\frac{(j+|\nu|-\frac 12)!}{j!(-j+|\nu|-\frac 12)!}(4\lambda \mu_i)^{-j} \left[\frac{(a/b)^{(p-k)/2}}{K_{p-k}(\sqrt{ab})} x^{p-k-1} e^{-(ax+ b/x)/2}\right].
	\end{align*}
	Note that as $|\nu| - \frac{1}{2}=k-1$, the second inequality in the statement of the lemma follows using the fact that $ |\Delta_i\mathtt{Y}|^{-j-k}\leq C(1+ |\Delta_i\mathtt{Y}|^{-2k+1})$ for all $ j=0,...,\lfloor{|\nu|-\frac 12}\rfloor$. 
	%
	\bibliographystyle{abbrv}
	\bibliography{bibli} 
\end{document}